\documentclass[reqno,twoside,12pt]{amsart}

\hoffset - 1.5cm

\usepackage{amsmath}
\usepackage{amssymb}
\usepackage{graphicx}
\usepackage{amstext}
\usepackage{amsopn}
\usepackage{amsfonts}
\usepackage{xcolor}
\usepackage{setspace}
\usepackage{enumerate}
\usepackage[polish,english]{babel}
\usepackage[utf8]{inputenc}
\usepackage{polski}

\usepackage{lmodern}

\newtheorem{Remark}{Remark}[section]
\newtheorem{Corollary}[Remark]{Corollary}
\newtheorem{Definition}[Remark]{Definition}
\newtheorem{Example}[Remark]{Example}
\newtheorem{Fact}[Remark]{Fact}
\newtheorem{Lemma}[Remark]{Lemma}
\newtheorem{Proposition}[Remark]{Proposition}
\newtheorem{Theorem}[Remark]{Theorem}

\newcommand{\ba}{\begin{array}}
\newcommand{\bc}{\begin{center}}
\newcommand{\bd}{\begin{description}}
\newcommand{\bdm}{\begin{displaymath}}
\newcommand{\be}{\begin{enumerate}}
\newcommand{\beq}{\begin{equation}}
\newcommand{\bdf}{\begin{Definition}}
\newcommand{\bex}{\begin{Example}}
\newcommand{\bft}{\begin{Fact}}
\newcommand{\bl}{\begin{Lemma}}
\newcommand{\bp}{\begin{Proposition}}
\newcommand{\br}{\begin{Remark}}
\newcommand{\bt}{\begin{Theorem}}
\newcommand{\bco}{\begin{Corollary}}
\newcommand{\bhy}{\begin{Hypothesis}}
\newcommand{\ea}{\end{array}}
\newcommand{\ec}{\end{center}}
\newcommand{\ed}{\end{description}}
\newcommand{\edm}{\end{displaymath}}
\newcommand{\ee}{\end{enumerate}}
\newcommand{\eeq}{\end{equation}}
\newcommand{\edf}{\end{Definition}}
\newcommand{\eex}{\end{Example}}
\newcommand{\eft}{\end{Fact}}
\newcommand{\el}{\end{Lemma}}
\newcommand{\ep}{\end{Proposition}}
\newcommand{\er}{\end{Remark}}
\newcommand{\et}{\end{Theorem}}
\newcommand{\eco}{\end{Corollary}}
\newcommand{\ehy}{\end{Hypothesis}}

\newcommand{\bC}{\mathbb{C}}

\newcommand{\bH}{\mathbb{H}}
\newcommand{\bI}{\mathbb{I}}

\newcommand{\bN}{\mathbb{N}}

\newcommand{\bR}{\mathbb{R}}

\newcommand{\bT}{\mathbb{T}}

\newcommand{\bV}{\mathbb{V}}
\newcommand{\bW}{\mathbb{W}}

\newcommand{\bZ}{\mathbb{Z}}

\newcommand{\cC}{\mathcal{C}}

\newcommand{\cH}{\mathcal{H}}

\newcommand{\cM}{\mathcal{M}}

\newcommand{\cO}{\mathcal{O}}

\newcommand{\cW}{\mathcal{W}}

\numberwithin{equation}{section} \errorcontextlines=0

\newcommand{\rank}{\mathrm{rank}\:}

\newcommand{\sub}{\overline{\mathrm{sub}}}

\newcommand{\SO}{\operatorname{SO}}
\newcommand{\Id}{\operatorname{Id}}
\newcommand{\BIF}{\mathcal{BIF}}

\makeatletter
\@namedef{subjclassname@2020}{%
  \textup{2020} Mathematics Subject Classification}
\makeatother

\usepackage{todonotes}
\usepackage{setspace}

\begin{document}

\title[Unbounded sets of solutions]{Unbounded sets of solutions of non-cooperative elliptic systems on symmetric spaces}

\author{Piotr Stefaniak}
\address{Faculty of Mathematics and Computer Science \\ Nicolaus Copernicus University in Toru\'n\\
PL-87-100 Toru\'{n} \\ ul. Chopina $12 \slash 18$ \\
Poland}

\email{cstefan@mat.umk.pl}

\date{\today}


\begin{abstract}
The aim of this paper is to show that, for a class of non-cooperative elliptic systems on compact symmetric spaces, any continuum of nontrivial solutions bifurcating from the set of trivial solutions is unbounded. The main tool is the degree for invariant strongly indefinite functionals. The analysis relies on the torus-equivariant structure of the Laplace--Beltrami eigenspaces. The result is obtained by ruling out return to the trivial branch in an equivariant version of the Rabinowitz alternative.
\end{abstract}

\maketitle
\section{Introduction}

In this article we study connected sets of weak solutions of nonlinear elliptic systems of the form
\begin{equation}\label{eq:intr-main}
\left\{\begin{array}{lclcl}
a_1\Delta_M u_1(x)&=&\nabla_{u_1}F(u(x),\lambda) \\
a_2\Delta_M u_2(x)&=&\nabla_{u_2}F(u(x),\lambda) \\
&\vdots&&\ \ \ \text{on}& M, \\
a_p\Delta_M u_p(x) &=&\nabla_{u_p}F(u(x),\lambda)
\end{array}
\right.
\end{equation}
where $M=G/H$ is a compact symmetric space. Here $G$ is a compact connected semisimple Lie group, $H\subset G$ is the fixed-point subgroup of an involutive automorphism of $G$ and $\Delta_M$ is the Laplace--Beltrami operator on $M$. The constants $a_i\in\{-1,1\}$ allow for an indefinite sign structure in the system and $F\in C^2(\bR^p\times\bR,\bR)$ is a nonlinear potential depending on a real parameter $\lambda\in\bR$. We assume that $F$ satisfies suitable structural and growth conditions, stated later as assumptions \textup{(A1)}-\textup{(A4)}. In particular, $0\in\bR^p$ is a critical point of $F(\cdot,\lambda)$ for every $\lambda\in\bR$, so that $u\equiv 0$ is a trivial solution of the system.

To study \eqref{eq:intr-main} we use a variational approach. We associate to the system a functional $\Phi\colon \bH\times\bR\to\bR$ whose critical points correspond to weak solutions. The corresponding gradient map $\nabla_u\Phi$ is a completely continuous perturbation of a Fredholm operator on a Hilbert space $\bH$, which carries a natural orthogonal action of the torus $\bT$ of dimension equal to $\rank M$, arising from the geometry of the symmetric space $M$. The functional $\Phi$ is $\bT$-invariant and its gradient $\nabla_u\Phi$ is $\bT$-equivariant.

In the absence of symmetry, continua of critical points of variational problems are governed by the classical Rabinowitz alternative, which, under suitable parity assumptions, ensures that each bifurcating continuum is either unbounded or meets another characteristic value, see \cite{Nirenberg, Rabinowitz1971, Rabinowitz1986}. In the presence of group symmetries, however, the topological methods underlying this alternative may no longer apply. A fundamental result of Dancer shows that the Leray--Schauder degree of a $G$-equivariant completely continuous perturbation of the identity depends only on its restriction to the fixed-point subspace, see \cite{Dancer1982}. This illustrates the general failure of the classical degree to detect bifurcation phenomena in equivariant problems. Broader formulations of this principle have been developed both for the Brouwer and Leray--Schauder degrees, see \cite{KusBal}, for the abelian case in particular see \cite{IzeVig}.

To overcome this difficulty, one is led to seek topological invariants adapted to equivariant problems. Equivariant degree theories for general maps have been developed in various settings, see for example \cite{IzeVig, KusBal}. In the variational framework relevant here one works with gradient maps, but taking the gradient structure alone does not lead to a genuinely stronger topological invariant. Indeed, a result of Parusiński (see~\cite{Parusinski}) shows that two gradient maps are homotopic if and only if they are gradient-homotopic. Thus, without additional structure, a degree theory defined only for gradient maps cannot capture more information than the classical Brouwer or Leray--Schauder degree. This is why one incorporates group symmetries into the framework.

Building on this idea, further developments have focused on constructing an equivariant degree for compact Lie group actions. In the finite-dimensional setting such a degree has been introduced by Gęba in \cite{Geba} and later extended to infinite-dimensional Hilbert spaces by Rybicki in \cite{Ryb2005milano} for equivariant gradient maps that are completely continuous perturbations of the identity. A degree for $G$-invariant strongly indefinite functionals has been developed in~\cite{GolRyb} and a method for computing it in the presence of a critical orbit of possibly non-zero dimension has been given in~\cite{GS2020}. For general treatments of equivariant degree theory we refer the reader to~\cite{BKS}.

An alternative invariant for detecting bifurcation is the (equivariant) Conley index. An equivariant version suitable for strongly indefinite functionals has been developed by Izydorek in~\cite{Izydorek} and applies in particular to the class of systems considered here. While the Conley index is highly effective at confirming the existence of bifurcation points, it remains essentially local and does not provide information about the global geometry of bifurcating sets - for instance, whether continua are unbounded or how distinct branches may be connected. These global properties, which are central to our analysis, require different topological invariants such as the equivariant degree.

Motivated by this, we study global bifurcation phenomena for nonlinear elliptic systems on compact symmetric spaces using the equivariant degree as the primary tool. While bifurcation theory in the Euclidean setting is relatively well developed, corresponding results on Riemannian manifolds, particularly for systems, are less common. In particular, the unboundedness of continua of solutions has been established only in a limited number of cases. As examples we mention several results obtained on the sphere and related domains.
In~\cite{RybickiLB} global bifurcation from the trivial branch $\{0\} \times \bR$ and unbounded continua have been established for a single equation on the sphere and this analysis has been extended to non-cooperative systems in~\cite{RybSte2015}. A related result has been given in~\cite{RybShiSte}, where elliptic systems have been studied on geodesic balls in the sphere and unboundedness has been proved in particular for hemispheres. The work~\cite{GolSte2021} has considered elliptic systems on the sphere with an invariant potential, allowing for bifurcation from an orbit of stationary solutions and has also established unboundedness of the resulting continua.

The present article extends the bifurcation results of~\cite{RybickiLB} and~\cite{RybSte2015} from the standard sphere, which is a compact symmetric space of rank one, to arbitrary compact symmetric spaces. In the rank-one case the argument relies on an $S^1$-action and the corresponding equivariant degree. In the general case we replace this by the action of the torus $\bT$. While an $S^1$-action would be sufficient to detect bifurcation values, the full torus action carries additional information on the representation structure of Laplace--Beltrami eigenspaces and this makes it possible to prove unboundedness of the bifurcating continua. Theorem~\ref{thm:unbound} shows that, with the possible exception of $\lambda=0$, every parameter value at which bifurcation is expected from the linearized problem actually gives rise to bifurcation from the trivial branch and that the corresponding continua of nontrivial solutions are unbounded. Thus the result provides not only existence of global bifurcation, but also qualitative information about the global structure of the solution set.

The proof combines an equivariant Rabinowitz alternative with an explicit computation of the $\bT$-equivariant bifurcation index. The key ingredients are the description of Laplace--Beltrami eigenspaces in terms of irreducible $G$-representations with given restricted highest weights, the analysis of their restricted weight structure and auxiliary calculations in the Euler ring $U(\bT)$. In particular, Proposition~\ref{prop:first-plane} shows that the real two-dimensional torus representation associated with the restricted highest weight $\alpha$ appears for the first time in the eigenspace corresponding to $\lambda_\alpha$ and with multiplicity one, while Lemma~\ref{lem:multinUT} gives the corresponding coefficient of the bifurcation index explicitly.

The article is organized as follows. Section~\ref{sec:prelim} collects the representation-theoretic, spectral and topological tools used later. Subsection~\ref{sub:weights} recalls weights, highest weights and their restricted counterparts. Subsection~\ref{subsec:laplacian} describes the Laplace--Beltrami spectrum on compact symmetric spaces and the induced torus representation structure of its eigenspaces. Subsection~\ref{subsec:degree} reviews the degree for invariant strongly indefinite functionals and records the auxiliary calculations in the Euler ring of $\bT$ needed in the bifurcation analysis.

Section~\ref{sec:main} contains the main results. We restate the elliptic system in the variational framework, identify the possible bifurcation values, define the $\bT$-equivariant bifurcation index and state an equivariant Rabinowitz alternative in Theorem~\ref{thm:AltRab}. We then compute the relevant coefficients of the bifurcation index and use them to prove Theorem~\ref{thm:unbound}. The final remarks discuss symmetry breaking and explain how the previous results for the sphere are recovered as a special case.

\section{Preliminary results}\label{sec:prelim}

In this section we collect the background material and notation used in the proof of the main theorem.

\subsection{Weights and highest weights}\label{sub:weights}
This subsection recalls the notions of weights, highest weights and their restricted counterparts used later. For background and further details, see \cite{BrockerDieck, FultonHarris, Gurarie, KnappLie}.

Throughout the paper we assume that $G$ is a compact connected semisimple Lie group.
By a representation of $G$ we mean a real or complex vector space $\bV$ equipped with a smooth linear action of $G$. We write $g\cdot v$ for the action of $g\in G$ on $v\in\bV$. A real $G$-representation is called orthogonal if $G$ acts by orthogonal linear maps. A representation is called irreducible if it has no proper nontrivial $G$-invariant subspaces.

Let $T$ be a maximal torus of $G$ and let $\frak t$ be its Lie algebra.
Consider the exponential map $\exp\colon\frak t\to T$ and put $\Lambda=\ker(\exp)$.
Then every $t\in T$ is of the form $t=\exp(\phi)$ for some $\phi\in\frak t$.
Let $\bV$ be a finite-dimensional complex representation of $G$. By a weight of $\bV$ we understand $\lambda\in\Lambda^*:=\{\lambda\in\frak t^*:\lambda(\Lambda)\subset2\pi\bZ\}\subset\frak t^*$
such that the corresponding weight space
\[
\bV_\lambda := \{ v \in \bV : \exp(\phi) \cdot v = e^{i \lambda(\phi)} v \text{ for all } \phi\in\frak t\}
\]
is non-zero.
Any finite-dimensional complex representation $\bV$ of $G$ admits a weight decomposition with respect to $T$, see for example \cite{BrockerDieck},
\begin{equation}\label{eq:weight-decomp}
\bV=\bigoplus_{\lambda\in\Lambda^*}\bV_\lambda.
\end{equation}

Now let $R\subset\frak t^*$ be the root system of $G$ with respect to $T$ and fix a positive system $R^+\subset R$. This determines a partial order on $\Lambda^*$ and if $\bV$ is irreducible, then among its weights there exists a unique maximal one. It is called the highest weight of $\bV$.

Following \cite[Subsection~5.7.2]{Gurarie}, we recall the restricted root system and the corresponding notions of restricted weights and restricted highest weights.
From now on $H\subset G$ is the fixed-point subgroup of an involutive automorphism $\sigma\colon G\to G$, that is, $\sigma^2=id$ and $H=\{g\in G:\sigma(g)=g\}$. Let $\frak g$ be the Lie algebra of $G$ and let $\frak g=\frak h\oplus\frak m$ be the Cartan decomposition, where $\frak h$ is the Lie algebra of $H$. Let $\frak a\subset\frak m$ be a maximal abelian subalgebra. Without loss of generality we assume that $\frak a\subset\frak t$. Set $\bT=\exp(\frak a)\subset T$. Then $\bT$ is a torus of dimension $r=\dim\frak a$.
The restricted root system is defined by
\[
\Sigma:=\{\alpha_{|\frak a}:\alpha\in R, \alpha_{|\frak a}\neq0\}\subset\frak a^*.
\]
Fix a decomposition $\Sigma=\Sigma^+\cup\Sigma^-$ into positive and negative restricted roots
and denote by $\{\alpha_1,\dots,\alpha_r\}\subset\Sigma^+$ the corresponding set of simple restricted roots. Then $\{\alpha_1,\dots,\alpha_r\}$ is a basis of $\frak a^*$.

Let $\bV$ be a finite dimensional complex representation of $G$. Put
$\Lambda_{\frak a}^*:=\{\lambda_{|\frak a}:\lambda\in\Lambda^*\}\subset\frak a^*$.
A restricted weight of $\bV$ is an element $\mu=\lambda_{|\frak a}\in\Lambda_{\frak a}^*$, where $\lambda\in\Lambda^*$ is a weight of $\bV$. For such $\mu$ we define the corresponding restricted weight space by
\begin{equation}\label{eq:restricted-weight-space}
\bV_\mu:=\bigoplus_{\lambda_{|\frak a}=\mu}\bV_\lambda.
\end{equation}
Let
$
\Lambda_{\frak a}^+:=\{\sum_{j=1}^r k_j\alpha_j:k_j\in\bN_0\}\subset\frak a^*
$
and define a partial order on restricted weights by putting $\mu\preceq\nu$ whenever $\nu-\mu\in\Lambda_{\frak a}^+$. We write $\mu\prec\nu$ if $\mu\preceq\nu$ and $\mu\neq\nu$.
If $\bV$ is irreducible and $\lambda\in\Lambda^*$ is its highest weight, then the restriction $\lambda_{|\frak a}\in\frak a^*$ is called the restricted highest weight of $\bV$. In particular, it is maximal among the restricted weights of $\bV$ with respect to the above order.

\subsection{Laplace--Beltrami spectrum on symmetric spaces}\label{subsec:laplacian}

Set $M=G/H$ and endow $M$ with the $G$-invariant Riemannian metric induced by the negative Killing form on $\frak g$. Then $M$ is a compact symmetric space of rank $r$.
Our aim in this subsection is to collect spectral properties of the Laplace--Beltrami operator on $M$. In particular, we describe the eigenspaces as $G$-representations and discuss their decompositions with respect to the torus $\bT$.
For background on compact symmetric spaces see \cite{HelgasonDiffGeom} and for the spectral theory of the Laplace--Beltrami operator see \cite[Section~5.7]{Gurarie}.

Let $\bV_{-\Delta_M}(\lambda)\subset L^2(M;\bR)$ denote the eigenspace of $-\Delta_M$ corresponding to the eigenvalue $\lambda$ in the space of real-valued functions. We denote by $\bV_{-\Delta_M}(\lambda)^\bC\subset L^2(M;\bC)$ its complexification, that is,
$
\bV_{-\Delta_M}(\lambda)^\bC=\bV_{-\Delta_M}(\lambda)\otimes_\bR\bC.
$
For a complex-valued function $u$, the equation $-\Delta_Mu=\lambda u$ holds if and only if $\operatorname{Re}u$ and $\operatorname{Im}u$ satisfy the same equation. Therefore $\bV_{-\Delta_M}(\lambda)^\bC$ is precisely the eigenspace of $-\Delta_M$ corresponding to $\lambda$ in the space of complex-valued functions.

Put
$\rho:=\frac12\sum_{\gamma\in\Sigma^+}m_\gamma\gamma,$
where $m_\gamma$ denotes the multiplicity of the restricted root $\gamma\in\Sigma^+$.
For $\alpha\in\Lambda_{\frak a}^+$ define
$\lambda_\alpha:=(\alpha+\rho,\alpha+\rho)-(\rho,\rho),$
where $(\cdot,\cdot)$ is the inner product on $\frak a^*$ induced by the Killing form.
The following result summarizes the spectral description that we will use, see \cite[Thm. 5.7.2]{Gurarie} and the remark following it.

\begin{Theorem}\label{thm:HC-spectrum}
The following hold.
\begin{enumerate}
\item For each $\alpha\in\Lambda_{\frak a}^+$ there exists an irreducible complex $G$-subrepresentation $\cH_\alpha\subset L^2(M;\bC)$ with restricted highest weight $\alpha$ such that $-\Delta_Mu=\lambda_\alpha u$ for every $u\in\cH_\alpha$.
\item Each $\cH_\alpha$ occurs in $L^2(M;\bC)$ with multiplicity $1$.
\item Every eigenvalue of $-\Delta_M$ is of the form $\lambda_\alpha$ for some $\alpha\in\Lambda_{\frak a}^+$.
\item For each eigenvalue $\lambda_\alpha$ one has
\[
\bV_{-\Delta_M}(\lambda_\alpha)^\bC=\bigoplus_{\beta\in\Lambda_{\frak a}^+:\lambda_\beta=\lambda_\alpha}\cH_\beta.
\]
\item If $\alpha\prec\beta$, then $\lambda_\alpha<\lambda_\beta$.
\end{enumerate}
\end{Theorem}

\begin{Example}\label{ex:sphere}
Let $M=S^n$ with the standard metric. Then $M=G/H$ with $G=\SO(n+1)$ and $H=\SO(n)$. In this case $r=\rank M=1$, so if $\alpha$ denotes the unique simple restricted root, then $\Lambda_{\frak a}^+=\{k\alpha:k\in\bN_0\}$. One has $\rho=\frac{n-1}{2}\alpha$ and
$
\lambda_{k\alpha}=k(k+n-1)
$
for $k\in\bN_0$, see \cite[Section 5.7.8, Eq. (7.15)]{Gurarie}. The corresponding real eigenspace is $\bV_{-\Delta_M}(\lambda_{k\alpha})=\cH_k^n$, where $\cH_k^n$ denotes the space of real spherical harmonics of degree $k$ on $S^n$. The complexification $\cH_k^n\otimes_\bR\bC$ is an irreducible $\SO(n+1)$-representation, hence $\cH_k^n$ is irreducible as a real $\SO(n+1)$-representation.
\end{Example}

In rank one symmetric spaces the eigenspaces of $-\Delta_M$ are irreducible, as illustrated in Example \ref{ex:sphere}. This follows from the fact that for $r=1$ the parameter $\alpha$ is determined by a single nonnegative integer and the map $\alpha\mapsto\lambda_\alpha$ is strictly increasing, so each eigenvalue corresponds to exactly one restricted highest weight.
For higher rank symmetric spaces this is no longer true, since distinct restricted highest weights can yield the same eigenvalue. The next example illustrates this phenomenon.

\begin{Example}\label{ex:product-spheres}
Let $G=\SO(n_1+1)\times\SO(n_2+1)$ and $H=\SO(n_1)\times\SO(n_2)$, so that $M=G/H=S^{n_1}\times S^{n_2}$ with the product of the standard metrics. By Example~\ref{ex:sphere}, each factor has rank $1$, hence $r=\rank M=2$. Thus for some simple restricted roots $\alpha_1,\alpha_2$, one has
$
\Lambda_{\frak a}^+=\{k\alpha_1+l\alpha_2:k,l\in\bN_0\}.
$
By the product structure of $M$ and Example~\ref{ex:sphere},
$
\rho=\frac{n_1-1}{2}\alpha_1+\frac{n_2-1}{2}\alpha_2.
$
Moreover, for $k,l\in\bN_0$,
\[
\lambda_{k\alpha_1+l\alpha_2}=k(k+n_1-1)+l(l+n_2-1)
\]
and
\[
\cH_{k\alpha_1+l\alpha_2}=\left(\cH^{n_1}_{k}\otimes\cH^{n_2}_{l}\right)\otimes_\bR\bC\subset L^2(M;\bC)
\]
is the irreducible complex $G$-representation from Theorem \ref{thm:HC-spectrum}. If $n_1=n_2$, then $\lambda_{k\alpha_1+l\alpha_2}=\lambda_{l\alpha_1+k\alpha_2}$ and for $k\ne l$
\[
\bV_{-\Delta_M}(\lambda_{k\alpha_1+l\alpha_2})^\bC=\cH_{k\alpha_1+l\alpha_2}\oplus\cH_{l\alpha_1+k\alpha_2},
\]
hence $\bV_{-\Delta_M}(\lambda_{k\alpha_1+l\alpha_2})^\bC$ is not an irreducible $G$-representation.
\end{Example}

\begin{Lemma}\label{lem:highest}
For every $\alpha\in\Lambda_{\frak a}^+$, the restricted weight space $(\cH_\alpha)_\alpha$ is one-dimensional.
\end{Lemma}

\begin{proof}
Consider the trivial representation of $H$ on $\bC$. Identifying its induced $G$-rep\-re\-sen\-ta\-tion with $L^2(G/H;\bC)$, Frobenius reciprocity yields, see \cite[Thm.~9.9]{KnappLie}, that the multiplicity of the irreducible $G$-representation $\cH_\alpha$ in $L^2(G/H;\bC)$ equals $\dim (\cH_\alpha)^H$.
By Theorem \ref{thm:HC-spectrum}(2), this multiplicity is equal to $1$, hence $\dim (\cH_\alpha)^H=1$. In particular, $\cH_\alpha$ has a non-zero $H$-fixed vector.

By Theorem \ref{thm:HC-spectrum}(1), the restricted highest weight of $\cH_\alpha$ is equal to $\alpha$. Therefore, applying the Helgason theorem in the form stated in \cite[Thm.~8.49]{KnappLie}, we conclude that the $\bT$-weight space corresponding to the restricted highest weight of $\cH_\alpha$ is one-dimensional. Hence the restricted highest weight $\alpha$ occurs in $\cH_\alpha$ with multiplicity $1$.
\end{proof}

We now describe the real eigenspaces as real $\bT$-representations.
Let $\mu\in\Lambda_{\frak a}^*$. For $\mu\neq0$ we define $\bR[1,\mu]$ to be the real orthogonal two-dimensional representation of $\bT$ in which $t=\exp(\phi)\in\bT$ acts on $\bR^2$ by rotation through the angle $\mu(\phi)$. Explicitly,
\[
\bR[1,\mu]:=\left(\bR^2, \exp(\phi)\mapsto
\begin{pmatrix}
\cos\mu(\phi)&-\sin\mu(\phi)\\
\sin\mu(\phi)&\cos\mu(\phi)
\end{pmatrix}\right).
\]
For $\mu=0$ we put $\bR[1,0]:=\bR$, the trivial representation. For $k\in\bN$ we define
$
\bR[k,\mu]:=\bigoplus_{j=1}^k\bR[1,\mu].
$
By \cite[Proposition II.8.5]{BrockerDieck} the irreducible real orthogonal representations of $\bT$ are all of this form, including the trivial case $\mu=0$. For $\mu\neq 0$ the representations $\bR[1,\mu]$ and $\bR[1,-\mu]$ are isomorphic.

\begin{Remark}\label{rem:weights-integers}
After fixing the basis $\{\alpha_1,\dots,\alpha_r\}$ of simple restricted roots, every element $\mu$ of their integer span can be written uniquely in the form
$
\mu=\sum_{j=1}^r m_j\alpha_j
$
with $(m_1,\dots,m_r)\in\bZ^r$. Thus one may identify $\mu\in \Lambda_{\frak a}^*$ with the tuple $m=(m_1,\dots,m_r)\in\bZ^r$ and in this notation $\bR[1,\mu]$ becomes $\bR[1,m]$. This is the coordinate form of the notation used for instance in \cite{GarRyb}.
\end{Remark}

In the next section we will need to understand the real Laplace--Beltrami eigenspaces as $\bT$-representations. Theorem \ref{thm:HC-spectrum} provides a decomposition of each complexified eigenspace into the irreducible summands $\cH_\alpha$, each occurring with multiplicity $1$. This determines where the representations $\bR[1,\alpha]$ can occur and shows that, for $\alpha\neq0$, such a representation occurs for the first time at the eigenvalue $\lambda_\alpha$.

\begin{Proposition}\label{prop:first-plane}
For every $\alpha\in\Lambda_{\frak a}^+\setminus\{0\}$ the real $\bT$-representation $\bR[1,\alpha]$ occurs with multiplicity $1$ in the eigenspace $\bV_{-\Delta_M}(\lambda_\alpha)$. Moreover, if $\beta\in\Lambda_{\frak a}^+$ and $\lambda_\beta<\lambda_\alpha$, then $\bR[1,\alpha]$ does not occur in $\bV_{-\Delta_M}(\lambda_\beta)$.
\end{Proposition}

\begin{proof}
Fix $\alpha\in\Lambda_{\frak a}^+\setminus\{0\}$ and let $\lambda_\alpha$ be the corresponding eigenvalue. Using \eqref{eq:weight-decomp} and \eqref{eq:restricted-weight-space}, the complex eigenspace $\bV_{-\Delta_M}(\lambda_\alpha)^\bC$, regarded as a complex $\bT$-representation, admits the decomposition into restricted weight spaces
\[
\bV_{-\Delta_M}(\lambda_\alpha)^\bC=\bigoplus_{\mu\in \Lambda_{\frak a}^*}\bV_\mu.
\]
Since $\bV_{-\Delta_M}(\lambda_\alpha)^\bC$ is the complexification of a real $\bT$-representation, complex conjugation satisfies $\overline{\bV_\mu}=\bV_{-\mu}$. In particular, $\bV_\alpha\oplus\bV_{-\alpha}$ is invariant under conjugation, hence it is the complexification of its real part
$
\{u\in \bV_\alpha\oplus\bV_{-\alpha}:\overline u=u\}
$
which is a $\bT$-invariant subspace of $\bV_{-\Delta_M}(\lambda_\alpha)$ and, as a real $\bT$-representation, is isomorphic to $\bR[\dim\bV_\alpha,\alpha]$. Therefore the multiplicity of $\bR[1,\alpha]$ in $\bV_{-\Delta_M}(\lambda_\alpha)$ is equal to $\dim\bV_\alpha$.

By Theorem \ref{thm:HC-spectrum}(4),
\[
\bV_{-\Delta_M}(\lambda_\alpha)^\bC=\bigoplus_{\gamma\in\Lambda_{\frak a}^+:\lambda_\gamma=\lambda_\alpha}\cH_\gamma.
\]
By Lemma \ref{lem:highest}, $\dim(\cH_\alpha)_\alpha=1$. Suppose that $\alpha$ also occurs in some $\cH_\gamma$ with $\lambda_\gamma=\lambda_\alpha$ and $\gamma\neq\alpha$. Since $\gamma$ is the restricted highest weight of $\cH_\gamma$, one has $\mu\preceq\gamma$ for every restricted weight $\mu$ of $\cH_\gamma$. In particular, as $\alpha$ is a restricted weight of $\cH_\gamma$ and $\alpha\neq\gamma$, we get $\alpha\prec\gamma$. By Theorem \ref{thm:HC-spectrum}(5), this implies $\lambda_\alpha<\lambda_\gamma$, which contradicts $\lambda_\gamma=\lambda_\alpha$. Therefore $\dim\bV_\alpha=1$ and consequently $\bR[1,\alpha]$ occurs with multiplicity $1$ in $\bV_{-\Delta_M}(\lambda_\alpha)$.

Now let $\beta\in\Lambda_{\frak a}^+$ be such that $\lambda_\beta<\lambda_\alpha$. If $\bR[1,\alpha]$ occurred in $\bV_{-\Delta_M}(\lambda_\beta)$, then the restricted weight $\alpha$ would occur in $\bV_{-\Delta_M}(\lambda_\beta)^\bC$. By Theorem \ref{thm:HC-spectrum}(4), it would then occur in some $\cH_\gamma$ with $\lambda_\gamma=\lambda_\beta$. Since $\lambda_\gamma<\lambda_\alpha$, we have $\gamma\neq\alpha$ and as above this implies $\alpha\prec\gamma$. By Theorem \ref{thm:HC-spectrum}(5), this gives $\lambda_\alpha<\lambda_\gamma=\lambda_\beta$, a contradiction. Hence $\bR[1,\alpha]$ does not occur in $\bV_{-\Delta_M}(\lambda_\beta)$.
\end{proof}

\begin{Remark}
Having the highest weight of an irreducible representation, all other weights occurring in the representation can be determined for example via Kostant's multiplicity formula, see \cite[Chapter VI, Thm. 3.2]{BrockerDieck}. Restricting these weights to $\frak a$ and using \eqref{eq:restricted-weight-space}, one obtains the corresponding restricted weights and their multiplicities, keeping in mind that distinct weights may have the same restriction to $\frak a$. Combined with the argumentation from the proof, this determines the decomposition of $\bV_{-\Delta_M}(\lambda_\alpha)$ as a $\bT$-representation.
\end{Remark}

For a real $G$-representation $\bV$ denote by $\bV^G$ its fixed point subspace.

\begin{Proposition}\label{prop:fixed-point}
If $\lambda_\alpha\neq 0$, then $\bV_{-\Delta_M}(\lambda_\alpha)^G=\{0\}$.
\end{Proposition}
\begin{proof}
The action of $G$ on $M$ is transitive, so for any $x,y\in M$ there exists $g\in G$ such that $y=gx$. Let $u\in\bV_{-\Delta_M}(\lambda_\alpha)^G$. Then $u(y)=u(gx)=u(x)$, hence $u$ is constant and $\Delta_M u=0$. Since $-\Delta_M u=\lambda_\alpha u$ and $\lambda_\alpha\neq 0$, it follows that $u=0$.
\end{proof}

\subsection{The $\nabla_\bT$-degree for strongly indefinite functionals}\label{subsec:degree}
We recall basic properties of the Euler ring of a torus, which will be used in the construction of an equivariant bifurcation index. For the general theory of the Euler ring $U(G)$ of a compact Lie group $G$ we refer to \cite{TomDieck1979,TomDieck}. Here we restrict to the case $G=\bT$ and follow the notation from~\cite{GarRyb}.

Let $\sub(\bT)$ denote the set of closed subgroups of $\bT$. For $H\in\sub(\bT)$, consider the homogeneous space $\bT/H$ with the left $\bT$-action induced by left translation. We write $\bT/H^+$ for the pointed space obtained by adjoining a disjoint base point. It is known that $\bT/H^+$ is a pointed $\bT$-CW-complex with exactly one $0$-cell. The $\bT$-equivariant Euler characteristic of $\bT/H^+$, denoted $\chi_\bT(\bT/H^+)$, is defined as in \cite[Definition 2.3]{GarRyb}.

The Euler ring $U(\bT)$ is the free $\bZ$-module generated by the elements $\chi_\bT(\bT/H^+)$ with $H\in\sub(\bT)$, equipped with the multiplication induced by the smash product of pointed $\bT$-CW-complexes.
For tori, an explicit multiplication formula is given in \cite{GarRyb}. Namely, if $H,H'\in\sub(\bT)$ and $H'':=H\cap H'$, then
\begin{equation}\label{eq:UTmultiplication}
\chi_\bT(\bT/H^+)\star\chi_\bT(\bT/H'^+)=
\begin{cases}
\chi_\bT(\bT/H''^+), & \dim H+\dim H'=\dim\bT+\dim H'',\\
\Theta, & \dim H+\dim H'<\dim\bT+\dim H'',
\end{cases}
\end{equation}
where $\Theta$ is the zero element and the unit is $\bI:=\chi_\bT(\bT/\bT^+)$.

\begin{Remark}\label{rem:bigcodim} From equation \eqref{eq:UTmultiplication} it follows that if $y\in U(\bT)$ is a $\bZ$-linear combination of terms $\chi_\bT(\bT/H^+)$ with $\dim H\le \dim\bT-2$, then for every $x\in U(\bT)$ the product $x\star y$ is again a $\bZ$-linear combination of such terms.
\end{Remark}

The Euler ring $U(\bT)$ is used to define the degree for $\bT$-invariant strongly indefinite functionals introduced in~\cite{GolRyb}, which we now briefly recall.

Let $\bH$ be a real Hilbert space equipped with an orthogonal action of $\bT$. Assume that $\bH$ admits an approximation scheme $\{\pi_n:\bH\to\bH\}_{n\in\bN_0}$ in the sense of \cite{GolRyb}, and write $\bH^n:=\pi_n(\bH)$ for each $n$. Consider a $\bT$-invariant functional $\Phi\in C^1(\bH,\bR)$ of the form $\Phi(u)=\tfrac12\langle Lu,u\rangle-\eta(u)$, where:
\begin{itemize}
\item $L:\bH\to\bH$ is a bounded, self-adjoint, $\bT$-equivariant Fredholm operator of index $0$, with $\ker L=\bH^0$ and $\pi_n\circ L=L\circ\pi_n$ for all $n$,
\item $\nabla\eta:\bH\to\bH$ is a completely continuous, $\bT$-equivariant operator.
\end{itemize}
Here $\bT$-invariance of $\Phi$ means that $\Phi(t\cdot u)=\Phi(u)$ for all $t\in\bT$ and $u\in\bH$, so that $\nabla\Phi$ is $\bT$-equivariant in the sense that $\nabla\Phi(t\cdot u)=t\cdot\nabla\Phi(u)$.

Let $\Omega\subset\bH$ be an open, bounded, $\bT$-invariant subset such that $\partial\Omega\cap(\nabla\Phi)^{-1}(0)=\emptyset$. Then the degree for $\bT$-invariant strongly indefinite functionals, $\nabla_\bT\text{-}\deg(\nabla\Phi,\Omega)\in U(\bT)$, is well-defined. It satisfies the standard properties: excision, additivity, homotopy invariance, multiplicativity and linearization at nondegenerate critical points.

The degree of the negative identity map is given by
\begin{equation}\label{eq:degree-id}
\nabla_\bT\text{-}\deg(-\mathrm{Id},B(\bV))=\chi_\bT(S^\bV),
\end{equation}
where $B(\bV)$ denotes the unit ball in a finite-dimensional orthogonal $\bT$-representation $\bV$, and $S^\bV=B(\bV)/\partial B(\bV)$ is its one-point compactification, see \cite{Geba}.

For every nonzero restricted weight $\mu\in\Lambda_{\frak a}^*$ set
$H_\mu:=\{\exp(\phi)\in\bT:\mu(\phi)\in2\pi\bZ\}.$
The group $H_\mu$ is the isotropy subgroup of the $\bT$-action on $\bR[1,\mu]$ at any nonzero point. By \cite[Corollary 4.6.7]{DuisKolk} and the classification of real $\bT$-representations, the subgroups $H_\mu$ are precisely the closed subgroups of codimension one in $\bT$ and every proper closed subgroup of $\bT$ is an intersection of finitely many such subgroups.

Since for $\mu\neq0$ the real $\bT$-representations $\bR[1,\mu]$ and $\bR[1,-\mu]$ are isomorphic and $H_\mu=H_{-\mu}$, we fix a subset $\cM$ of the set of nonzero restricted weights such that every nonzero element of $\Lambda_{\frak a}^+$ belongs to $\cM$ and for every nonzero $\mu\in\Lambda_{\frak a}^*$ exactly one of $\mu$ and $-\mu$ belongs to $\cM$.

In what follows, $\cO$ denotes an unspecified $\bZ$-linear combination of terms $\chi_\bT(\bT/H^+)$ with $\dim H\le\dim\bT-2$. By Remark~\ref{rem:bigcodim}, such terms remain of the same form under multiplication in $U(\bT)$.
The following theorem is a reformulation of \cite[Theorem 3.2]{GarRyb} in the notation introduced above. It expresses $\chi_\bT(S^\bV)$ in terms of the subgroups $H_\mu$.
\begin{Theorem}\label{thm:UTnmult}
Let
$\bV=\bR[k_0,0]\oplus\bigoplus_{\mu\in\cM}\bR[k_\mu,\mu]$
be a $\bT$-representation, where $k_0,k_\mu\in\bN_0$ and $k_\mu=0$ for all but finitely many $\mu\in\cM$. Then
\[
\chi_\bT(S^\bV)=(-1)^{k_0}\left(\bI-\sum_{\mu\in\cM} k_\mu\chi_\bT(\bT/H_\mu^+)\right)+\cO.
\]
\end{Theorem}

Using \eqref{eq:UTmultiplication} and Theorem~\ref{thm:UTnmult} one can check that
\begin{equation}\label{eq:chi-inv-codim1}
\chi_\bT(S^\bV)^{-1}=(-1)^{k_0}\left(\bI+\sum_{\mu\in\cM} k_\mu\chi_\bT(\bT/H_{\mu}^+)\right)+\cO.
\end{equation}

For $x\in U(\bT)$ and $H\in\sub(\bT)$, let $x_H$ denote the coefficient of $\chi_\bT(\bT/H^+)$ in $x$.

\begin{Lemma}\label{lem:codim1-coeff}
Consider $\bT$-representations of the form
\[
\bV=\bR[k_0,0]\oplus\bigoplus_{\mu\in\cM}\bR[k_\mu,\mu],\ \
\bW=\bR[l_0,0]\oplus\bigoplus_{\mu\in\cM}\bR[l_\mu,\mu],
\]
where $k_0,l_0,k_\mu,l_\mu\in\bN_0$ and $k_\mu=l_\mu=0$ for all but finitely many $\mu\in\cM$. Then for every $\mu\in\cM$ and $N\in\bN$ one has
\begin{equation}\label{eq:codim1-product}
\bigl(\chi_\bT(S^\bV)^N\star(\chi_\bT(S^\bW)^N-\bI)\bigr)_{H_\mu}=(-1)^{k_0N}N\Bigl((-1)^{l_0N+1}l_\mu-\bigl((-1)^{l_0N}-1\bigr)k_\mu\Bigr)
\end{equation}
and
\begin{equation}\label{eq:codim1-product-inv}
\bigl(\chi_\bT(S^\bV)^{-N}\star(\chi_\bT(S^\bW)^N-\bI)\bigr)_{H_\mu}=(-1)^{k_0N}N\Bigl((-1)^{l_0N+1}l_\mu+\bigl((-1)^{l_0N}-1\bigr)k_\mu\Bigr).
\end{equation}
\end{Lemma}

\begin{proof}
By Theorem~\ref{thm:UTnmult} we have
\[
\chi_\bT(S^\bV)=(-1)^{k_0}\left(\bI-\sum_{\mu\in\cM}k_\mu\chi_\bT(\bT/H_\mu^+)\right)+\cO.
\]
Using \eqref{eq:UTmultiplication} and Remark~\ref{rem:bigcodim}, one obtains
\[
\chi_\bT(S^\bV)^N=(-1)^{k_0N}\left(\bI-N\sum_{\mu\in\cM}k_\mu\chi_\bT(\bT/H_\mu^+)\right)+\cO.
\]
Similarly,
\[
\chi_\bT(S^\bW)^N-\bI=\bigl((-1)^{l_0N}-1\bigr)\bI-(-1)^{l_0N}N\sum_{\mu\in\cM}l_\mu\chi_\bT(\bT/H_\mu^+)+\cO.
\]

Multiplying these expansions, Remark~\ref{rem:bigcodim} shows that the terms $\cO$ do not contribute to the coefficient at $H_\mu$. Moreover, by \eqref{eq:UTmultiplication}, products of two codimension one terms do not contribute to a codimension one coefficient. Taking the coefficient at $H_\mu$ in the product yields \eqref{eq:codim1-product}.

For \eqref{eq:codim1-product-inv}, use \eqref{eq:chi-inv-codim1} and the same argument to obtain
\[
\chi_\bT(S^\bV)^{-N}=(-1)^{k_0N}\left(\bI+N\sum_{\mu\in\cM}k_\mu\chi_\bT(\bT/H_\mu^+)\right)+\cO
\]
and then take the coefficient at $H_\mu$ in
$\chi_\bT(S^\bV)^{-N}\star(\chi_\bT(S^\bW)^N-\bI)$.

\end{proof}

\begin{Remark}\label{rem:codim1-dim}
Since $\bR[k_\mu,\mu]$ is even-dimensional for $\mu\neq0$, we have $(-1)^{k_0}=(-1)^{\dim\bV}$ and $(-1)^{l_0}=(-1)^{\dim\bW}$. Thus the factors $(-1)^{Nk_0}$ and $(-1)^{Nl_0}$ in \eqref{eq:codim1-product} and \eqref{eq:codim1-product-inv} can be replaced by $(-1)^{N\dim\bV}$ and $(-1)^{N\dim\bW}$.
\end{Remark}

\section{Main Results}\label{sec:main}
In this section we formulate and prove the main unboundedness result for global bifurcation continua of the system \eqref{eq:main}. We introduce the variational setting, identify the bifurcation values, define an equivariant bifurcation index and use it to show that, with the possible exception of the zero level, all expected bifurcations from the trivial branch occur and give unbounded continua.

Consider the non-cooperative elliptic system on the compact symmetric space $M=G/H$:
\begin{equation}\label{eq:main}
a_i\Delta_M u_i(x)=\nabla_{u_i}F(u(x),\lambda),\quad i=1,\dots,p,\ x\in M,
\end{equation}
where:
\begin{itemize}
\item[(A1)] $F\in C^2(\bR^p\times\bR,\bR)$ with gradient $\nabla_uF=(\nabla_{u_1}F,\dots,\nabla_{u_p}F)$,
\item[(A2)] $\nabla_uF(u,\lambda)=\lambda u+\nabla_u h(u,\lambda)$ and $h(u,\lambda)=o(|u|^2)$ as $u\to 0$,
\item[(A3)] $a_i\in\{-1,1\}$,
\item[(A4)] there exist $q<2\dim M/(\dim M-2)$ (or any $q<\infty$ if $\dim M\le 2$) and  $C>0$ such that $|\nabla_u h(u,\lambda)|\le C(1+|u|^{q-1})$ for all $(u,\lambda)\in\bR^p\times\bR$.
\end{itemize}

Consider $\bH:=\bigoplus_{i=1}^p H^1(M)$  with the scalar product
\[
\langle u,v\rangle_{\bH}=\sum_{i=1}^p\langle u_i,v_i\rangle_{H^1}
=
\sum_{i=1}^p\int_M\bigl(\langle \nabla u_i,\nabla v_i\rangle+u_iv_i\bigr)dM
\]
and with the $\bT$-action $(t\cdot u)(x):=u(t^{-1}x)$. Then $\bH$ is an orthogonal $\bT$-representation.
 Define the $\bT$-invariant functional $\Phi:\bH\times\bR\to\bR$ by
\[
\Phi(u,\lambda):=-\tfrac12\sum_{i=1}^p a_i\int_M|\nabla u_i|^2 dM -\int_M F(u,\lambda) dM .
\]
Weak solutions of \eqref{eq:main} are exactly the critical points of $\Phi$.

Define the operator $T:H^1(M)\to H^1(M)$ by
$
\langle Tu,v\rangle_{H^1}=\int_M uv dM$  for all $u,v\in H^1(M).$
Then $T$ is bounded, self-adjoint, compact and $\bT$-equivariant. Set
$
\eta_0(u,\lambda):=\int_M h(u(x),\lambda) dM(x)
$
and define operators $L,K:\bH\to\bH$ by
$$
L(u_1,\dots,u_p):=(-a_1u_1,\dots,-a_pu_p),\
K(u_1,\dots,u_p):=(Tu_1,\dots,Tu_p).$$
Then
\[
\nabla_u\Phi(u,\lambda)=Lu-(\lambda\Id+L)Ku-\nabla_u\eta_0(u,\lambda).
\]
Moreover, $L$ is a bounded, self-adjoint and $\bT$-equivariant Fredholm operator of index $0$, while $K$ and $\nabla_u\eta_0(\cdot,\lambda)$ are $\bT$-equivariant and completely continuous. The former is immediate from the definition of $T$. The latter follows from assumption \textup{(A4)}, the compact Sobolev embedding $H^1(M)\hookrightarrow L^q(M)$, see \cite{Aubin,Hebey}, and the standard argument from \cite{Rabinowitz1986}.

We call a solution $(u,\lambda)\in\bH\times\bR$ of $\nabla_u\Phi(u,\lambda)=0$ trivial if $u=0$ and nontrivial otherwise. The set of trivial solutions is thus given by $\{0\} \times \bR$. A point $(0,\lambda_0)\in\bH\times\bR$ is called a global bifurcation point of nontrivial solutions of $\nabla_u\Phi(u,\lambda)=0$ if there exists a closed and connected set $\cC(\lambda_0)\subset\bH\times\bR$ such that
\begin{enumerate}
\item $\cC(\lambda_0)\subset closure\{(u,\lambda)\in\bH\times\bR:\nabla_u\Phi(u,\lambda)=0, u\neq 0\}$,
\item $(0,\lambda_0)\in\cC(\lambda_0)$ and $\cC(\lambda_0)\neq\{(0,\lambda_0)\}$,
\item either $\cC(\lambda_0)$ is unbounded in $\bH\times\bR$ or it intersects the trivial branch at $(0,\lambda_1)$ with $\lambda_1\ne\lambda_0$.
\end{enumerate}

From the Implicit Function Theorem it follows that if $(0,\lambda_0)$ is a global bifurcation point, then $\ker\nabla_u^2\Phi(0,\lambda_0)\neq\{0\}$. We therefore set
$$
\Lambda=\{\lambda\in\bR:\ker\nabla_u^2\Phi(0,\lambda)\neq\{0\}\},
$$
the set of parameters for which the linearization at $(0,\lambda)$ is not invertible and thus the set of possible bifurcation levels.

From \textup{(A2)} it follows that
\[
\nabla_u^2\Phi(0,\lambda)=(-a_1\Delta_M-\lambda \mathrm{Id},\ldots,-a_p\Delta_M-\lambda \mathrm{Id}).
\]
The following lemma is an immediate consequence. Denote by $\sigma(-\Delta_M)$ the spectrum of $-\Delta_M$ and set $-\sigma(-\Delta_M)=\{-\mu:\mu\in\sigma(-\Delta_M)\}$. Let $n_\pm$ be the number of indices $i\in\{1,\dots,p\}$ with $a_i=\pm1$ in \eqref{eq:main}.
\begin{Lemma}\label{lem:spec-structure}
With notation as above:
\begin{enumerate}[(i)]
\item The set $\Lambda$ is discrete and
\[
\Lambda=
\begin{cases}
\sigma(-\Delta_M), & n_->0, n_+=0,\\
-\sigma(-\Delta_M), & n_+>0, n_-=0,\\
\sigma(-\Delta_M)\cup-\sigma(-\Delta_M), & n_+n_->0.
\end{cases}
\]
\item Fix $\lambda_{\alpha}\in\sigma(-\Delta_M)\setminus\{0\}$. Then:
\begin{enumerate}[(a)]
\item $\ker\nabla_u^2\Phi(0,\lambda_{\alpha})=\bigoplus_{i=1}^{n_-}\bV_{-\Delta_M}(\lambda_{\alpha})$;
\item $\ker\nabla_u^2\Phi(0,-\lambda_{\alpha})=\bigoplus_{i=1}^{n_+}\bV_{-\Delta_M}(\lambda_{\alpha})$;
\item $\ker\nabla_u^2\Phi(0,0)=\bigoplus_{i=1}^{p}\bV_{-\Delta_M}(0)$.
\end{enumerate}
\end{enumerate}
\end{Lemma}

The next theorem is the main result of this paper. It shows that, with the possible exception of $\lambda=0$, every parameter value at which bifurcation is expected actually gives rise to bifurcation from the trivial branch. Moreover, it establishes that the corresponding continua of nontrivial solutions are unbounded in $\bH\times\bR$, thus providing not only existence but also qualitative information about the global structure of the solution set.

\begin{Theorem}\label{thm:unbound}
Consider the system \eqref{eq:main} with potential $F$ satisfying assumptions \textup{(A1)}-\textup{(A4)} and let $\lambda_\alpha\in\sigma(-\Delta_M)\setminus\{0\}$.
\begin{enumerate}
\item If $n_->0$, then $(0,\lambda_\alpha)$ is a global bifurcation point and the continuum $\cC(\lambda_\alpha)\subset\bH\times\bR$  is unbounded.
\item If $n_+>0$, then $(0,-\lambda_\alpha)$ is a global bifurcation point and the continuum $\cC(-\lambda_\alpha)\subset\bH\times\bR$  is unbounded.
\item If $p$ is odd, then $(0,0)$ is a global bifurcation point and the continuum $\cC(0)\subset\bH\times\bR$ is unbounded.
\end{enumerate}
\end{Theorem}

To prove Theorem \ref{thm:unbound} we use the $\bT$-equivariant degree for strongly indefinite functionals and the associated bifurcation index, defined as follows.
Fix $\lambda_0\in\Lambda$ and choose $\epsilon>0$ so that
$
[\lambda_0-\epsilon,\lambda_0+\epsilon]\cap\Lambda=\{\lambda_0\}.
$
Then
$
\partial B_\delta(\bH)\cap(\nabla_u\Phi(\cdot,\lambda_0\pm\epsilon))^{-1}(0)=\emptyset
$
for some $\delta>0$,
where $B_\delta(\bH)$ denotes the ball in $\bH$ of radius $\delta$ centered at $0$. In particular, the degrees
$
\nabla_\bT\text{-}\deg(\nabla_u\Phi(\cdot,\lambda_0\pm\epsilon),B_\delta(\bH))
$
are well defined, with the standard approximation scheme defined in terms of the eigenspace decomposition.
The bifurcation index at $\lambda_0$ is then defined by
\[
\BIF_\bT(\lambda_0)=\nabla_\bT\text{-}\deg(\nabla_u\Phi(\cdot,\lambda_0+\epsilon),B_\delta(\bH))-\nabla_\bT\text{-}\deg(\nabla_u\Phi(\cdot,\lambda_0-\epsilon),B_\delta(\bH))\in U(\bT).
\]
The following theorem is an equivariant version of the Rabinowitz alternative formulated in terms of this bifurcation index.

\begin{Theorem}\label{thm:AltRab}
Fix $\lambda_0\in\Lambda$. Assume that either $\lambda_0\neq 0$ or $\lambda_0=0$ and $p$ is odd. Then $\BIF_\bT(\lambda_0)\neq\Theta$ and consequently $(0,\lambda_0)$ is a global bifurcation point. Moreover, one of the following alternatives holds:
\begin{enumerate}
\item $\cC(\lambda_0)$ is unbounded in $\bH\times\bR$,
\item $\cC(\lambda_0)$ is bounded and
\[
\cC(\lambda_0)\cap(\{0\}\times\bR)=\{0\}\times\{\lambda_1,\dots,\lambda_s\}\subset\{0\}\times\Lambda,
\]
with $\BIF_\bT(\lambda_1)+\cdots+\BIF_\bT(\lambda_s)=\Theta$.
\end{enumerate}
\end{Theorem}

The proof follows the classical argument of \cite{Rabinowitz1971}, with the Leray--Schauder degree replaced by $\nabla_\bT$-degree. The nonvanishing of $\BIF_\bT(\lambda_0)$ is ensured by \cite[Theorem 4.5]{GarRyb}. If $\lambda_0\neq 0$, this applies because $\bV_{-\Delta_M}(\lambda_0)$ is a nontrivial $\bT$-representation by Proposition~\ref{prop:first-plane}. If $\lambda_0=0$ and $p$ is odd, it applies because $\dim\ker\nabla_u^2\Phi(0,0)=p$ by Lemma~\ref{lem:spec-structure}.

For $\lambda_\alpha\in\sigma(-\Delta_M)\setminus\{0\}$ let $\cW(\lambda_\alpha)=\bigoplus_{\lambda_\beta<\lambda_\alpha}\bV_{-\Delta_M}(\lambda_\beta)$. Reasoning as in \cite[Lemma 3.4]{GolSte2021} we obtain:

\begin{Lemma}\label{lem:indices}
Fix $\lambda_\alpha\in\sigma(-\Delta_M)\setminus\{0\}$.
\begin{enumerate}
\item If $n_->0$, then
\[
\BIF_\bT(\lambda_\alpha)=\nabla_\bT\text{-}\deg(-\mathrm{Id},B(\cW(\lambda_\alpha)))^{n_-}\star\bigl(\nabla_\bT\text{-}\deg(-\mathrm{Id},B(\bV_{-\Delta_M}(\lambda_\alpha)))^{n_-}-\bI\bigr).
\]
\item If $n_+>0$, then
\[
\BIF_\bT(-\lambda_\alpha)=\nabla_\bT\text{-}\deg(-\mathrm{Id},B(\cW(\lambda_\alpha)\oplus\bV_{-\Delta_M}(\lambda_\alpha)))^{-n_+}\star\bigl(\nabla_\bT\text{-}\deg(-\mathrm{Id},B(\bV_{-\Delta_M}(\lambda_\alpha)))^{n_+}-\bI\bigr).
\]
\item If $\lambda_\alpha=0$, then $\BIF_\bT(0)=((-1)^{n_-}-(-1)^{n_+})\bI$.
\end{enumerate}
\end{Lemma}

To apply Lemma~\ref{lem:indices}, we now focus on the coordinate $\BIF_\bT(\pm\lambda_\alpha)_{H_\alpha}$, defined as the coefficient of $\chi_\bT(\bT/H_\alpha^+)$ in $\BIF_\bT(\pm\lambda_\alpha)\in U(\bT)$. The following lemma provides the explicit values of these coordinates.
\begin{Lemma}\label{lem:multinUT}
Fix $\lambda_\alpha,\lambda_\beta\in\sigma(-\Delta_M)\setminus\{0\}$ and set $d_W=\dim\cW(\lambda_\alpha)$ and $d_V=\dim\bV_{-\Delta_M}(\lambda_\alpha)$. Then
\begin{enumerate}[(i)]
\item if $\lambda_\beta<\lambda_\alpha$, then $\BIF_\bT(\pm\lambda_\beta)_{H_\alpha}=0$,
\item $\BIF_\bT(\lambda_\alpha)_{H_\alpha}=(-1)^{(d_W+d_V)n_-+1}n_-$,
\item $\BIF_\bT(-\lambda_\alpha)_{H_\alpha}=(-1)^{(d_W+d_V)n_++1}n_+$.
\end{enumerate}
\end{Lemma}

\begin{proof}
(i) Fix $\lambda_\beta<\lambda_\alpha$. By Proposition~\ref{prop:first-plane}, the representation $\bR[1,\alpha]$ is absent from $\bV_{-\Delta_M}(\lambda_\beta)$, hence it is absent from $\cW(\lambda_\beta)$ and from $\bV_{-\Delta_M}(\lambda_\beta)$. In the notation of Lemma~\ref{lem:codim1-coeff} this means that the multiplicities $k_\alpha$ and $l_\alpha$ equal $0$ for every factor appearing in the products from Lemma~\ref{lem:indices} at the level $\lambda_\beta$. Therefore \eqref{eq:codim1-product} and \eqref{eq:codim1-product-inv} give $\BIF_\bT(\pm\lambda_\beta)_{H_\alpha}=0$.

(ii) By Lemma~\ref{lem:indices}(i) we have
$$\BIF_\bT(\lambda_\alpha)=\chi_\bT(S^{\cW(\lambda_\alpha)})^{n_-}\star(\chi_\bT(S^{\bV_{-\Delta_M}(\lambda_\alpha)})^{n_-}-\bI).$$
By Proposition~\ref{prop:first-plane}, the multiplicity of $\bR[1,\alpha]$ equals $0$ in $\cW(\lambda_\alpha)$ and equals $1$ in $\bV_{-\Delta_M}(\lambda_\alpha)$. Applying \eqref{eq:codim1-product} with $\bV=\cW(\lambda_\alpha)$, $\bW=\bV_{-\Delta_M}(\lambda_\alpha)$, $N=n_-$ and $\mu=\alpha$ and using Remark~\ref{rem:codim1-dim}, yields
\[
\BIF_\bT(\lambda_\alpha)_{H_\alpha}=(-1)^{n_-d_W}n_-(-1)^{n_-d_V+1}=(-1)^{(d_W+d_V)n_-+1}n_-.
\]

(iii) By Lemma~\ref{lem:indices}(ii),
$$\BIF_\bT(-\lambda_\alpha)=\chi_\bT(S^{\cW(\lambda_\alpha)\oplus\bV_{-\Delta_M}(\lambda_\alpha)})^{-n_+}\star(\chi_\bT(S^{\bV_{-\Delta_M}(\lambda_\alpha)})^{n_+}-\bI).$$
Here $\bR[1,\alpha]$ occurs with multiplicity $1$ in $\cW(\lambda_\alpha)\oplus\bV_{-\Delta_M}(\lambda_\alpha)$ and with multiplicity $1$ in $\bV_{-\Delta_M}(\lambda_\alpha)$. Applying \eqref{eq:codim1-product-inv} with $\bV=\cW(\lambda_\alpha)\oplus\bV_{-\Delta_M}(\lambda_\alpha)$, $\bW=\bV_{-\Delta_M}(\lambda_\alpha)$, $N=n_+$ and $\mu=\alpha$, and using Remark~\ref{rem:codim1-dim}, we obtain
\[
\BIF_\bT(-\lambda_\alpha)_{H_\alpha}=(-1)^{n_+(d_W+d_V)}n_+\Bigl((-1)^{n_+d_V+1}+\bigl((-1)^{n_+d_V}-1\bigr)\Bigr)=(-1)^{(d_W+d_V)n_++1}n_+.
\]
\end{proof}

\subsubsection*{Proof of Theorem \ref{thm:unbound}}\label{subsec:proof}
Fix $\lambda_\alpha\in\sigma(-\Delta_M)\setminus\{0\}$ and assume $n_->0$. By Theorem~\ref{thm:AltRab}, there is a continuum $\cC(\lambda_\alpha)\subset\bH\times\bR$ emanating from $(0,\lambda_\alpha)$. Suppose that $\cC(\lambda_\alpha)$ is bounded. Then Theorem~\ref{thm:AltRab} yields $\lambda_1,\dots,\lambda_s\in\Lambda$ such that
\[
\cC(\lambda_\alpha)\cap(\{0\}\times\bR)=\{0\}\times\{\lambda_1,\dots,\lambda_s\}\subset\{0\}\times\Lambda\]
and
\[\BIF_\bT(\lambda_1)+\dots+\BIF_\bT(\lambda_s)=\Theta.
\]
Without loss of generality, assume $\lambda_\alpha=\max\{|\lambda_j|:1\le j\le s\}$. Taking the $H_\alpha$-coordinate, we get
\[
\BIF_\bT(\lambda_1)_{H_\alpha}+\dots+\BIF_\bT(\lambda_s)_{H_\alpha}=0.
\]
By Lemma~\ref{lem:multinUT}(i), all terms with $|\lambda_j|<\lambda_\alpha$ vanish in the $H_\alpha$-coordinate. Hence the above identity reduces to
\[
\theta_+\BIF_\bT(\lambda_\alpha)_{H_\alpha}+\theta_-\BIF_\bT(-\lambda_\alpha)_{H_\alpha}=0
\]
for some $\theta_\pm\in\{0,1\}$.
Lemma~\ref{lem:multinUT} gives
\[
\BIF_\bT(\lambda_\alpha)_{H_\alpha}=(-1)^{(d_W+d_V)n_-+1}n_-,\
\BIF_\bT(-\lambda_\alpha)_{H_\alpha}=(-1)^{(d_W+d_V)n_++1}n_+.
\]
Since $n_->0$, the first number is nonzero, so necessarily $\theta_+=\theta_-=1$. We obtain
\[
(-1)^{(d_W+d_V)n_-+1}n_-+(-1)^{(d_W+d_V)n_++1}n_+=0.
\]
Equivalently,
$
(-1)^{(d_W+d_V)(n_--n_+)}n_-=-n_+,
$
which is impossible for nonnegative integers $n_\pm$ with $n_->0$. Thus $\cC(\lambda_\alpha)$ is unbounded. This proves (1). The proof of (2) is analogous, starting from $n_+>0$ and the continuum $\cC(-\lambda_\alpha)$.

Finally assume $p$ is odd. By Lemma~\ref{lem:indices}(iii) one has $\BIF_{\bT}(0)\neq 0$, so by Theorem~\ref{thm:AltRab} the point $(0,0)$ is a global bifurcation point and there exists a continuum $\cC(0)\subset\bH\times\bR$ bifurcating from $(0,0)$.

Suppose that $\cC(0)$ is bounded. Then, by Theorem~\ref{thm:AltRab},
\[
\cC(0)\cap(\{0\}\times\bR)=\{0\}\times\{\lambda_1,\dots,\lambda_s\}\subset\{0\}\times\Lambda.
\]
If $\lambda_j\neq 0$, then the same connected set $\cC(0)$ also gives global bifurcation from $(0,\lambda_j)$. By the first part of the proof, every such continuum must be unbounded. This contradicts the boundedness of $\cC(0)$. Therefore $\cC(0)$ is unbounded.

This completes the proof of Theorem~\ref{thm:unbound}.
\qed

\begin{Remark}\label{rem:sym-break}
It is known that if $(0,\lambda_0)$ is a bifurcation point and $\ker\nabla_u^2\Phi(0,\lambda_0)^G=\{0\}$, then a symmetry breaking occurs at $(0,\lambda_0)$, meaning that there exists a neighbourhood of $(0,\lambda_0)$ such that every nontrivial solution on the bifurcating branch has strictly smaller isotropy group than the trivial solution, whose isotropy group equals $G$. This criterion goes back to Dancer \cite{Dancer} and its extension to bifurcation from orbits is given in \cite[Lemma 3.11]{GKS2018}.

In our setting, every bifurcation point $(0,\lambda_0)$ with $\lambda_0\neq 0$ is a symmetry breaking point. Indeed, if $\lambda_0\in\Lambda\setminus\{0\}$, then $\lambda_0=\pm\lambda_\alpha$ for some $\lambda_\alpha\in\sigma(-\Delta_M)\setminus\{0\}$. By Lemma \ref{lem:spec-structure}\textup{(ii)}, the space $\ker\nabla_u^2\Phi(0,\lambda_0)$ is a direct sum of copies of $\bV_{-\Delta_M}(\lambda_\alpha)$. By Proposition \ref{prop:fixed-point} we have $\bV_{-\Delta_M}(\lambda_\alpha)^G=\{0\}$, hence also $\ker\nabla_u^2\Phi(0,\lambda_0)^G=\{0\}$. Therefore symmetry breaking occurs at every nonzero bifurcation point.
\end{Remark}

\begin{Remark}
In the special case of the sphere, see Example~\ref{ex:sphere}, one has $\lambda_{k\alpha}=k(k+n-1)$ and $\bV_{-\Delta_M}(\lambda_{k\alpha})=\cH_k^n$, where $\cH_k^n$ is an irreducible $\SO(n+1)$-representation with restricted highest weight $k\alpha$. Thus the bifurcation results of \cite{RybickiLB} for a single equation and of \cite{RybSte2015} for systems are recovered as special cases of our main theorem. The present framework extends these results from spheres to arbitrary compact symmetric spaces.
\end{Remark}

\end{document}